\documentclass{article}
\usepackage[utf8]{inputenc}
\usepackage{xcolor}
\usepackage{amsfonts,amsmath,amsthm, amssymb}
\usepackage{geometry}
\usepackage{enumerate}
\usepackage[colorinlistoftodos]{todonotes}
\usepackage{hyperref}
\usepackage{comment}
\theoremstyle{plain}
\newtheorem{theorem}{Theorem}[section]

\newtheorem{lemma}[theorem]{Lemma}
\newtheorem{proposition}[theorem]{Proposition}

\theoremstyle{definition}

\newtheorem{remark}[theorem]{Remark}
\newtheorem{theoremA}{Theorem}

\def\txtd{{\textnormal{d}}}

\def\cO{\mathcal{O}}
\def\cA{\mathcal{A}}
\def\cF{\mathcal{F}}

\def\cC{\mathcal{C}}
\def\cL{\mathcal{L}}

\def\bR{\mathbb{R}}
\def\bP{\mathbb{P}}
\def\bE{\mathbb{E}}

\newcommand{\rmd}{\mathrm{d}}
\title{On the pitchfork bifurcation for the Chafee-Infante equation with additive noise}
\author{Alex Blumenthal\thanks{School of Mathematics, Georgia Institute of Technology, Atlanta, GA 30332, USA. E-Mail:~ablumenthal6@gatech.edu}~, Maximilian Engel\thanks{Department of Mathematics, Free University of Berlin, Arnimallee 6, 14195 Berlin, Germany.~E-Mail:~maximilian.engel@fu-berlin.de}~~and~Alexandra Neam\c tu\thanks{Department of Mathematics and Statistics, University of Konstanz,  Universit\"atsstr.~10, 78464 Konstanz, Germany. E-Mail:~alexandra.neamtu@uni-konstanz.de}}

\numberwithin{equation}{section}
\begin{document}
	
	\maketitle

	\begin{abstract}
%	Using a random dynamical systems approach, w
	We investigate pitchfork bifurcations for a stochastic reaction diffusion equation perturbed by an infinite-dimensional Wiener process. It is well-known that the random attractor is a singleton, independently of the value of the bifurcation parameter; this phenomenon is often referred to as the ``destruction'' of the bifurcation by the noise. Analogous to the results of [Callaway et al., AIHP Probab. Stat., 53:1548-1574, 2017] for a 1D stochastic ODE, 
we show that some remnant of the bifurcation persists for this SPDE 
model in the form of a positive finite-time Lyapunov exponent. Additionally, we prove finite-time expansion of volume with increasing dimension as the bifurcation parameter crosses further eigenvalues of the Laplacian.

%	 similar results by \cite{Callawayetal}, 
%
%	
%	 this work we show that some 
%remnant of these bifurcations persists in the form of a positive finite time Lyapunov exponent
%
%	
%	
%	when the bifurcation parameter crosses the first eigenvalue of the Laplacian, we show a change in the sign of finite-time Lyapunov exponents along the random attractor, indicating that something remains of the bifurcation for the deterministic system. 
%%	 and the expansion of volume with increasing dimension. 
%	 Hence, beyond global asymptotic synchronization, we detect a bifurcation pattern similar to the deterministic case.
	\end{abstract}

	
	% 60H15, spdes, 60H50 reg by noise, 37H20 bifurcations for rds, 37L55 infinite dimensional rds
	\maketitle

	\section{Introduction} 
%	The impact of stochastic noise on the behaviour of a dynamical system is an intensely studied topic of mathematical and physical research.
%	Very often the mathematical analysis focusses on the statistics of trajectories for different noise realisations and does not consider the dynamical aspects of the system. In contrast, the theory of random dynamical systems as coined by the works of Ludwig Arnold and his co-workers in the 1980s and 1990s, and manifested in Arnold's book \textit{Random Dynamical Systems} \cite{Arnoldbook}, compares trajectories with different initial conditions but driven by the same noise. A random dynamical system in this sense consists of a model of the time-dependent noise seen as a dynamical system $\theta$ on the probability space, and a model of the dynamics on the state space formalized as a \textit{cocycle} $\varphi$ over $\theta$.

	We study bifurcations for the reaction diffusion equation known as the \emph{Chafee-Infante equation}, under perturbations by infinite-dimensional additive noise. The Chafee-Infante equation without noise, in close relation to other reaction diffusion systems with cubic nonlinearity such as the Allen-Cahn or the Nagumo equation, is a well-studied parabolic partial differential equation (PDE) with global attractor whose bifurcation behaviour is fully understood as a cascade of pitchfork bifurcations. We employ the viewpoint of random dynamical systems theory (see e.g.~\cite{Arnoldbook}) to detect a similar bifurcation pattern for the noisy case.
	%\todo[inline]{@All: This sounds more like an abstract right now, I know. But we can see how to rearrange this with the abstract and a convenient entry point for the intro at the end.}
	
	In more detail, we consider the following stochastic partial differential equation (SPDE) with Dirichlet boundary conditions on a bounded domain $\cO\subset\mathbb{R}$, say $\cO=[0,L]$,
	\begin{align}\label{spde_intro}
	\begin{cases}
	\txtd u = (\Delta u + \alpha u - u^3) ~\txtd t + \sqrt{Q}\txtd W_t, \\
	u(0)=u_{0}\in H, \quad u|_{\partial \cO} =0,
	\end{cases}
	\end{align}
	where $\alpha \in \mathbb R$ is the deterministic bifurcation parameter, $H:=L^2(\cO)$ is the state space and $(W_t)_{t \in \mathbb{R}}$ denotes a two-sided $H$-cylindrical Wiener process, with covariance operator $Q$ as specified in Section~\ref{sec:prelim}.

	In the case without noise (see e.g.~\cite{Robinson}), 
 the Chafee-Infante equation is well-posed, yielding a semigroup $S(t)$ on $H$ for which one can find a global attractor $A$, i.e., $A$ is a compact invariant subset of $H$ ($S(t) A = A$ for all $t \geq 0$) which attracts the orbits of all bounded subsets of $H$. 
	Starting with the homogeneous zero solution for $\alpha < \lambda_1$, where $\lambda_i$ denote the eigenvalues of the Laplacian $-\Delta$, an unstable direction and two new stationary solutions are added whenever $\alpha$ passes $\lambda_n$; the attractor then consists of these stationary points together with their respective unstable manifolds.
	
	In the case with noise, the attractor becomes a random object, a so-called \emph{random attractor} $A(\omega)$, whose position in the state space depends on the noise realization. It has been shown in \cite[Section 6]{Caraballoetal} that,
if $Q$ is bounded and invertible with bounded inverse (e.g. space-time white-noise), for any value of $\alpha$, the random attractor $A(\omega)$ of~\eqref{spde_intro} consists of a single point almost surely, i.e., there exists a random variable $a = a_\alpha:\Omega\to H$ with
	\begin{align*}
	\varphi(t,\omega,a(\omega))=a(\theta_t\omega)\quad \mbox{ for every } t\geq 0 \quad \mathbb{P}-a.s.,
	\end{align*}
	such that $A(\omega)=\{a(\omega)\}$. This phenomenon is often  called \emph{synchronization by noise} and has been studied  thoroughly in recent years for finite-dimensional SODEs \cite{CrauelFlandoli, Tearne, FlandoliGessScheutzowPTRF} and SPDEs \cite{FlandoliGessScheutzowAnnProb, Caraballoetal, Bloemker}. 
 The proof of synchronization for \eqref{spde_intro} in \cite[Section 6]{Caraballoetal}, adapting ideas from \cite{CrauelFlandoli}, uses the correspondence between stationary measures and attractors and a monotonicity argument. These ideas on synchronization carry over in our setting, where the covariance operator $Q$ of the noise will be given by a negative fractional power of the Laplacian, in order to ensure a suitable regularity of the solution according to Section~\ref{sec:prelim}. In a similar spirit, results on synchronization for~\eqref{spde_intro} with Neumann-boundary conditions (and $
\alpha=0$) have been derived in~\cite{scheutzow} provided that the noise is $H^{2s}(\cO)$-valued for $s>\frac{1}{4}$. Finally, for~\eqref{spde_intro} driven by space-time white-noise in two and three dimensions, a similar synchronzation phenomenon has been observed in \cite{GessT}.

	% The proof technique using the correspondence to a unique stationary measure and order-preservation was then transferred to the SPDE~\eqref{spde_intro} driven by a cylindrical Wiener process in  Section 6 of \cite{Caraballoetal}. For~\eqref{spde_intro} with singular noise a similar phenomenon has been observed in \cite{GessT}. Moreover,  results regarding synchronization by noise for higher-order SPDEs such as the Swift-Hohenberg equation are available in \cite{Bloemker} (note: this is a case where monotonicity arguments are not directly applicable). 

	Synchronization can be interpreted in the sense that the noise ``destroys'' the pitchfork bifurcation, since the attractor $A(\omega) = \{ a(\omega)\}$ remains a single point for all $\alpha \in \mathbb R$. This interpretation was embraced by Crauel and Flandoli \cite{CrauelFlandoli}  for the stochastically-forced pitchfork normal form on $\mathbb{R}$ given by
	\begin{equation} \label{eq:SODE}
	\rmd x = (\alpha x - x^3) \rmd t + \rmd W_t \, .
	\end{equation}
 There, they showed that the random attractor $A(\omega) \subset \mathbb{R}$ consists of a point for all values of $\alpha$, hence all trajectories synchronize to a random equilibrium. It follows that at the level of the asymptotic dynamics, the bifurcation at $\alpha = 0$ is ``destroyed'' and no switch to local instability occurs. 
% {\color{blue} Indeed, for this model it is possible to explicitly compute the stationary measure and asymptotic Lyapunov exponent $\Lambda_1$ of \eqref{eq:SODE} and to prove it is negative for all values of $\alpha \in \mathbb R$. {\color{red} reference?} It follows that nearby trajectories converge to each other, hence at the level of asymptotic dynamics, no switch to instability can be observed in the noisy case at value $\alpha = 0$. }

%	{\color{red} AB: Removed the bit about Lyapunov exponents of \eqref{eq:SODE}, which is already covered in a later remark and is not germane to our discussion here. }

	Callaway et al.~\cite{Callawayetal} challenged this point of view by measuring local stability for trajectories of the one-dimensional SODE model \eqref{eq:SODE} on finite time scales, using \emph{finite-time Lyapunov exponents} (FTLEs).  They proved that, whereas all FTLEs are negative for $\alpha < 0$, there is always a positive probability to observe positive FTLEs for $\alpha >0$; \emph{this change of in FTLEs corresponds with a transition from uniform to non-uniform attractivity of the attractor and a loss of (uniform) hyperbolicity}. %{\color{red} AB: removed a sentence here; removed mention of ``sensitive dependence''- it's not quite that.}
	% In order to prove such statements, a detailed analysis on the location of the random attractor is required. 
	The authors also showed that there is no uniformly continuous topological conjugacy between the dynamics for negative and positive $\alpha$. A similar result was proved for stochastic Hopf bifurcations with additive noise in \cite{DoanEngeletal}.

An alternative but strongly related problem is to fix the parameter $\alpha > 0$ in~\eqref{eq:SODE} (or $\alpha > \lambda_1$ in~\eqref{spde_intro} respectively), introduce a coefficient $\epsilon > 0$ in front of the noise term $\txtd W_t$ in \eqref{eq:SODE} (or \eqref{spde_intro} respectively), and determine for $\epsilon \ll 1$ the extent to which the random motion resembles that of the deterministic ($\epsilon = 0$) system. For the latter, the bulk of initial conditions relax to one of two stable solutions, while when $\epsilon > 0$ these solutions are \emph{metastable}, with typical random trajectories transitioning from the vicinity of one to the other, reflecting the fact that the system admits a unique stationary measure. 
Large deviations estimates provide a way of estimating the typical timescale of these
transitions, thereby giving another perspective on the ``persistence'' of the 
bifurcation in the presence of noise. For the system considered in this paper, this 
metastability picture was studied in \cite{ barret2015sharp, berglund2013sharp}.

	\subsubsection*{Statement of results}
	
	The purpose of this manuscript is to extend the mentioned ideas from \cite{Callawayetal, DoanEngeletal} to the infinite-dimensional setting, demonstrating
	similar bifurcation behavior for the random dynamical system induced by the SPDE~\eqref{spde_intro}.
	Let again $ 0 < \lambda_1 < \lambda_2 < \dots$  denote the eigenvalues of $- \Delta$ on $\cO=[0,L]$ with Dirichlet boundary conditions. 
	%Furthermore, let us introduce the FTLEs along the attracting random equilibrium by
	Given an initial condition $u_0 \in H$ and a sample $\omega \in \Omega$, the FTLE at time $t$ along the trajectory $\varphi^t_\omega(u)$ is given by
	\begin{equation} \label{eq:Lambdai}
	% \Lambda_i(t, a(\omega)) = \frac{1}{t} \ln \sigma_i(D_{a(\omega)} \varphi^t_\omega),
	\Lambda_1(t;  \omega, u_0) = \frac1t \ln \| D_{u_0} \varphi^t_\omega\|_H
	\end{equation}
	where $D_{u_0} \varphi^t_\omega$ is the linear operator on $H$ obtained by Fr\'echet differentiating the cocycle $ \varphi^t_\omega$ at $u_0$, and $\| \cdot \|_H$ denotes the norm on $H$. 
	
In line with \cite{Callawayetal} and in concert with our goal of assessing changes in the attractivity of the singleton attractor $a(\omega)$ of \eqref{spde_intro}, we are primarily\footnote{See Remark \ref{rmk:smallInitialData} below for a discussion of FTLEs at more general initial data.} interested in the FTLE
	% \textcolor{blue}{In order to quantify a stochastic bifurcation for the SPDE~\eqref{spde_intro}, we are primarly interested in the dynamics along the singleton attractor. Therefore we consider the linearization along this object and introduce the notation}
	\[
	\Lambda_1(t; \omega) := \Lambda_1(t; \omega, a(\omega))
	\]
	 taken \emph{along the attractor $a(\omega)$.}

		% $\sigma_i$ denote the singular values in decreasing order, counted with multiplicity. Recall that 
	%for a bounded linear operator $A$ on $H$, we have $\sigma_1(A) = \| A \|$, the operator norm of $A$. 
	
	%{\color{red}To do: define LE at a point $u$ and tangent vector $v$; split presentation of 
	%multiple LE to a Theorem B later on. }
	
	% Our main result concerns bifurcations in the FTLE $\Lambda_1(t; \omega)$. 
	\begin{theoremA}\label{thm:topLEEst}
		The random dynamical system induced by the SPDE~\eqref{spde_intro} exhibits a bifurcation at $\alpha = 0$ in the following sense:
		\begin{enumerate}[(a)]
			\item Unconditionally, we have that for all $\alpha \in \mathbb R$,
			\[
			\Lambda_1(t; \omega) \leq \alpha - \lambda_1 \quad \text{with probability 1 for all }t \geq 0 \,. 
			\]
			In particular, $\Lambda_1(t; \omega) < 0$ with probability 1 for all $\alpha < \lambda_1$.  
			
			%For any $\alpha < \lambda_1$, we have that $\lambda_1(t; \omega) \leq \alpha - \lambda_1 < 0$ almost surely. 
			
			\item For any $\alpha > \lambda_1,  0 < \delta \ll \alpha - \lambda_1$ and $T > 0$, there is a positive-probability event $\cA \subset \Omega$ such that %the maximal FTLE along the attractor $a(\omega)$ is positive for times $ t \in [0, T]$ and $\omega \in \mathcal{A}$, i.e.
			%$$ \Lambda_1(t, a( \omega)) \geq \alpha- \lambda_1 > 0  \ \text{ for all  } \omega \in \cA.$$
			\[
			\Lambda_1(t; \omega) \geq  \alpha - \lambda_1 - \delta > 0 \quad \text{ for all } \omega \in \cA \, , t \in [0, T]  \,. 
			\]
			%\item Furthermore, for all $k \geq 1$ and $T > 0$, if $\alpha > \lambda_k$, we have that the sum of the $k$ largest FTLEs is bounded below by
			%$$ \sum_{i=1}^k \Lambda_i(t, a( \omega)) \geq \sum_{i=1}^k \alpha - \lambda_i \ \text{ for all  } \omega \in \cA, $$
			%where again $\mathcal A$ is some set with $\mathbb{P}(\mathcal{A}) > 0$.
		\end{enumerate}
	\end{theoremA}
	Statements (a) and (b) together imply loss of uniform hyperbolicity of the system at $\alpha =\lambda_1$ corresponding with the deterministic pitchfork bifurcation of the PDE. 
	
	Our second result concerns the cascade of deterministic bifurcations when $\alpha$ crosses more and more eigenvalues $\lambda_k$. We are able to capture this by finite-time expansion of $k$-dimensional volumes, as measured by the quantities
	\[
	V_k(t; \omega, u_0) := \frac1t \log \| \wedge^k D_{u_0} \varphi^t_\omega \|_{\wedge^k H} \,, \quad V_k(t; \omega) := V_k(t; \omega, a(\omega)) \,,
	\]
	$k \geq 1$. 
	Here, $\wedge^k$ denotes the $k$-fold wedge product of a linear operator; see Section \ref{subsubsec:wedges} for
	details. Equivalently, $V_k$ can be characterized by volume growth: for a bounded linear operator $A$ on $H$, we have the identity
	\[
	\| \wedge^k A \|_{\wedge^k H} = \max\{  |\det(A|_E)|  : E \subset H , \dim E = k\}
	\]
	relating $\| \wedge^k A\|_{\wedge^k H}$ with the maximal volume growth $A$ exhibits along a $k$-dimensional subspace of $H$. Here, $\det(A|_E)$ is the determinant of $A|_E : E \to A(E)$ regarded
	as a linear operator between $E$ and $A(E)$, and we follow the convention $\det(A|_E) = 0$ if 
	$\dim A(E) < \dim E$.  
	
	\begin{theoremA}\label{thm:volumeFTLE} Let $k \geq 1$ be arbitrary. 
		\begin{enumerate}[(a)]
			\item Unconditionally, we have that for all $\alpha \in \mathbb R$, 
			\[
			V_k(t; \omega) \leq \sum_{i = 1}^k (\alpha - \lambda_i) \quad \text{with probability 1 for all }t \geq 0 \,. 
			\]
			\item For any $\alpha > \frac{1}{k} \sum_{i = 1}^k \lambda_i, 0 < \delta \ll \sum_{i=1}^k (\alpha - \lambda_i)$ and $T > 0$, 
			there is a positive probability event $\cA \subset \Omega$ such that
			\[
			V_k(t; \omega) \geq  \sum_{i = 1}^k (\alpha - \lambda_i) - \delta \quad \text{ for all } \omega \in \cA \, , t \in [0, T]  \,. 
			\]
		\end{enumerate}
	\end{theoremA}

	%\todo[inline]{I wrote (c) in terms of sums but we should discuss the write phrasing. Further clarify meaning and distinctions of FTLEs?}
	%
	%\todo[inline]{Maybe possible to add estimate on asymptotic $\Lambda_1 < 0$, via the invariant measure $\nu$ which has a density with respect to a Gausssian measure $\mu$ on $H$, given by
	%    \begin{align}
	%        p(x)=K e^{-\frac{1}{2}U(x)},
	%    \end{align}
	%    where $K$ is a normalization constant and $U$ is the antiderivative of $\alpha u -u^3$, i.e.~$U(x) = \frac{\alpha x^2}{2} - \frac{ x^4}{4}$?}
	\subsubsection*{Remarks and comments on the proof}
	
	Estimation of FTLE for the stochastic ODE \eqref{eq:SODE} is straightforward and proceeds roughly as follows.  Analogously to our setting, FTLE are unconditionally $\leq \alpha$ for any value of $\alpha \in \mathbb R$. 
	When $\alpha > 0$, the origin $0$ is linearly unstable, yet
	one can find a positive probability event that the point attractor remains close to 0 on arbitrarily long timescales, accumulating a positive FTLE $\approx \alpha$. 
	
	Our proof in the SPDE setting \eqref{spde_intro} follows roughly the same lines: the events $\cA$ we
	construct in Theorems \ref{thm:topLEEst} and \ref{thm:volumeFTLE} steer the random point 
	attractor $a_\alpha(\omega)$ towards $0$, where the linearization is ``close'' to the 
	shifted heat semigroup $e^{(\alpha + \Delta) t}$ with singular values $e^{t (\alpha - \lambda_i)}, i \geq 1$. However, making this rigorous entails several 
	challenges not present in the finite-dimensional case, e.g.:
	\begin{enumerate}[(a)]
		\item Due to the nonlinear term in \eqref{spde_intro}, the linearization
		is only close to the semigroup $e^{(\alpha + \Delta) t}$
		when the point attractor $a_\alpha(\omega)$ is small in $\cC(\cO)$, not just $H = L^2(\cO)$ (c.f. Proposition \ref{prop:FTLE_1}). 
		% Controlling higher regularity of the point attractor requires ensuring 
		% taking advantage of \emph{parabolic regularity}, the tendency of dissipative parabolic problems such as \eqref{spde_intro} to \emph{regularize}
		% initial data in $H$ into higher-regularity spaces. 
		For details, see Proposition \ref{h1:attractor} and its proof in the Appendix. %{\color{red} AB: removed the bit about parabolic regularity, not used in this paper.}
		\item Even when the linearization is steered to be close to the semigroup $e^{(\alpha + \Delta) t}$, 
		lower bounds on FTLE do not follow from naive $L^2$ energy estimates. Instead, we derive a lower bound
		using a careful invariant cones argument delineated by an appropriate quadratic form $Q_\delta$
		on $H \times H$. This approach, inspired by techniques for ODEs \cite{lewowicz1980lyapunov},  avoids verifying the abstract criteria developed in~\cite{Amann_1995} in order to establish cone invariance for parabolic evolution operators. The lower bounds on $k$-dimensional volume growth are obtained by extending this technique to the wedge spaces $\Lambda^k H$. See  Section \ref{subsubsec:wedges} for details. 
		%	it 
		%	is a delicate matter to find lower bounds on FTLE, as even the semigroup $e^{(\alpha + \Delta) t}$ 
		%	
	\end{enumerate}
	
	% we encounter several challenges in the SPDE setting. Firstly, we need to control the $H_0^1$-norm of the random equilibrium $a_{\alpha}(\omega)$ in order to make sure that it stays in a neighbourhood of the origin for arbitrarily long times with positive probability. The corresponding result is stated in Proposition~\ref{h1:attractor} and a proof is given in the Appendix, using a-priori estimates on the solutions of the corresponding random PDE and a singular Gronwall inequality, involving the Mittag-Leffler function.
	%Furthermore, bounding the top FTLE from below (see Proposition~\ref{prop:FTLE_1}) is achieved by using an appropriate quadratic form $Q_{\delta}$ on $L^2 \times L^2$ via which we can define an invariant cone $C_{\delta}$ for the heat semigroup and its perturbation as given by the variational equation of~\eqref{spde_intro} along $a_{\alpha}(\omega)$.
	%The additional lower bounds on the singular values for $\alpha$ crossing $\lambda_k$, $k \geq 1$, as stated in Proposition~\ref{prop:singvalues}, are obtained by extending this technique to the wedge space $\Lambda^k L^2$.
	
	%{\color{red}``significant difficulties only in the infinite dimensional setting''}
	
	%Further relevance of the result ... asymptotic Lyapunov exponent ...

	\begin{remark}
		The collection of Lyapunov exponents usually refers to the set of asymptotic exponential growth rates
		realized by different tangent directions. By the Multiplicative Ergodic Theorem (see, e.g., \cite{oseledets1968multiplicative, ruelle1982characteristic} or the survey \cite{young2013mathematical}), the rates achieved are precisely asymptotic exponential growth rates of singular values, and so a natural way to consider `lower' \emph{finite-time Lyapunov exponents} is to use finite-time singular values: 
		\[
		\Lambda_k(t; \omega, u_0) := \frac1t \log \sigma_k(D_{u_0} \varphi^t_\omega)
		\]
		where for $k \geq 1$ we write $\sigma_k(A)$ for the $k$-th singular value of a bounded linear operator $A$ on $H$. It is straightforward to show that for such a bounded operator $A$, we have the identity
		$\| \wedge^k A\|_{\wedge^k H} = \prod_{i = 1}^k \sigma_i(A)$, hence
		\[
		V_k(t; \omega) = \sum_{i =1 }^k \Lambda_i(t; \omega) \, .
		\]
		For $T > 0, 0 < \delta \ll 1$, Theorem \ref{thm:volumeFTLE} implies there is a positive probability event $\cA \subset \Omega$ so that
		\[
		\alpha - \lambda_k - \delta \leq \Lambda_k(t; \omega) \leq \alpha - \lambda_k + \delta
		\quad \text{ for all } \omega \in \cA, t \in [0,T] \, ,
		\]
		suggesting as before a bifurcation occurring as $\alpha$ moves past $\lambda_k$. 
%		However, this analysis does not imply an \emph{unconditional} upper bound on the $\Lambda_k, k \geq 1$ as one has for the $V_k$ and $\Lambda_1 = V_1$. 
	\end{remark}

\begin{remark}
		Our main results concern estimates of \emph{finite-time} Lyapunov exponents on positive probability events, i.e.,
		fixing a time $T > 0$ and estimating $\Lambda_k(t, \omega)$ for $t \in [0,T]$ and for $\omega$
		drawn from a positive-probability set in $\Omega$. 
		As $t \to \infty$, it follows from the subadditive ergodic theorem 
		and uniqueness of the stationary measure for the Markov process defined by \eqref{spde_intro} 
		that the \emph{asymptotic Lyapunov exponents}
		\begin{align} \label{eq:asymptExp}
			\lim_{t \to \infty} \Lambda_k(t; \omega)
		\end{align}
		exist for each $k \geq 1$ and are deterministic (independent of $\omega$) with probability 1 (see, e.g., \cite{ruelle1982characteristic}). 
		Since the random attractor of \eqref{spde_intro} consists of a single point with full 
		probability for all values of the bifurcation parameter, it is likely that
		 the asymptotic exponents as in \eqref{eq:asymptExp} are all negative (or at least nonpositive). 
		Indeed, in \cite{CrauelFlandoli} the asymptotic Lyapunov exponent of the one-dimensional system \eqref{eq:SODE} was shown to be negative, using an explicit calculation, and in \cite{DoanEngeletal} a quantitative negative upper bound was derived for the two-dimensional extension to Hopf bifurcations in certain parameter regimes. Perhaps unsurprisingly, though, the arguments from these papers do not carry over to the infinite dimensional setting of our work. Providing such an estimate remains an open problem for future work. 
	\end{remark}

\begin{remark}\label{rmk:smallInitialData}
Our estimates of FTLEs along the random attractor, carried out in Section \ref{sec:proofs}, are relatively versatile and make no direct use of assumptions about the nondegeneracy of the covariance (as in \cite{Caraballoetal}), otherwise used to ensure uniqueness of stationary measures or synchronization by noise-- see Remark \ref{rmk:useOfCovariance} for further discussion. Moreover, it is not hard to check that our estimates carry over to trajectories initiated at any sufficiently small initial $u_0 \in H$. However, our proof as-is does not extend to arbitrary initial data: see Section \ref{sec:outlook} at the end of the paper for further discussion along these lines. 
\end{remark}

% {\color{red} Comment on \cite{lewowicz1980lyapunov}}

%	\todo[inline]{OK with this or should we say more? -AB }\todo[inline]{Maybe we should say more. AN}
	
	\subsubsection*{Structure of the paper}
	
	Section~\ref{sec:prelim} recalls regularity properties of the solution to equation~\eqref{spde_intro} and the generation of an associated random dynamical system with random attractor and sample measures corresponding with the unique stationary measure of~\eqref{spde_intro}.
	Section~\ref{sec:proofs} is dedicated to proving Theorems \ref{thm:topLEEst} and \ref{thm:volumeFTLE}, firstly taking care of estimating the top FTLE with the cone technique described above, and secondly bounding the $k$-volume growth rates via the wedge spaces $\wedge^k H$. In the Appendix we provide the proof of the crucial Proposition~\ref{h1:attractor}, yielding control of the random equilibrium in $V_\gamma = H^{2\gamma}_0(\cO)$ for $\gamma\in(\frac{1}{4},\frac{1}{2})$. For an outlook on future
	problems along the lines of the results in this paper, see Section \ref{sec:outlook}.
	
%{\color{blue}	\subsubsection*{Acknowldgements}
%	We would like to thank Christian Keuhn for helpful discussion and insights. 
%}	

	\section{Preliminaries} \label{sec:prelim}

	We study bifurcations for the following reaction diffusion equation with additive noise. Let  $\alpha>0$, set $H:=L^2(\cO)$ and $V:=H^1_0(\cO)$ for a bounded domain $\cO\in\mathbb{R}$, e.g. $\cO=[0,L]$, and consider again the SPDE~\eqref{spde_intro}
	\begin{align*}
	\begin{cases}
	\txtd u = (\Delta u + \alpha u - u^3) ~\txtd t + \txtd W_t, \\
	u(0)=u_{0}\in H, \quad u|_{\partial \cO} =0.
	\end{cases}
	\end{align*}
	%\todo[inline]{In the following, we are switching a lot between $V, H$ and $L^2, H_0^1$ notation. Maybe we can be more coherent...}
	%%%%%%%%%	
	\renewcommand{\cL}{\Delta}
	
	\subsection{Background and basic properties}\label{assumptions}

% \begin{itemize}
% 	\item basic notation, $V_\gamma$ spaces
% 	\item set out the assumptions on the covariance operator: $Q = (- \Delta)^{- 2 \beta}$ where $\beta \in (0, \frac14)$ is fixed throughout. 
% 	\item Lemma: $V_\gamma$ continuous version of the stochastic convolution; fixes the probability space $(\Omega, \mathcal F, \P)$ once and for all. 
% 	\item Well-posedness, differentiability and RDS formulation
% \end{itemize}

	{\bf Notations, assumptions, and basic results.} 
	 Throughout, we regard the Laplacian $\Delta$ as a closed linear operator on $H = L^2(\cO)$ with Dirichlet boundary conditions. It is well-known that this generates a compact, analytic $C_0$-semigroup $(S(t))_{t\geq 0}$ on $H$ and the domain of its fractional powers can be identified with fractional Sobolev spaces~\cite{Amann_1995}. 
	  More precisely we introduce the spaces $V_\gamma:= (D(-\cL)^\gamma, \langle \cdot, \cdot\rangle_{V_\gamma})$, where $\langle x ,y\rangle_{V_\gamma} =\langle (-\cL)^\gamma x, (-\cL)^\gamma y \rangle$ for $x,y\in V_\gamma$. Then we can identify the fractional power spaces with Sobolev spaces~\cite[Theorem 16.15]{Yagi}
	\begin{equation*}
	V_\gamma = \begin{cases} H^{2\gamma}(\cO), & \gamma \in [0,\frac{1}{4}),\\
	H^{2\gamma}_0(\cO), & \gamma\in(\frac{1}{4},1]\setminus\{\frac{3}{4}\}.
		\end{cases}
	\end{equation*} 
	Note that in particular, 
	$
	V_{\frac{1}{2}}=H^1_0(\cO)=D((-\cL)^{1/2}) \,
	$
	and that $V_\gamma$ is continuously embedded in 
	$\cC( \cO)$
	%$L^\infty(\cO)$ 
	for all 
	$\gamma > \frac14$ by the Sobolev embedding theorem~\cite[Theorem 1.36]{Yagi}.

In this paper, we will utilize a two-sided $H$-cylindrical Wiener process $(W_t)_{t\in\bR}$ with covariance operator $Q$ (to be specified shortly)
	on a probability space $(\Omega, \cF, \bP)$. It is well-known that there exists a Hilbert space $\tilde{H}$ such that $H\subset \tilde{H}$ where $(W_t)_{t\in\mathbb{R}}$ is trace-class. Since we are working in a random dynamcial systems framework, it is expedient to use the canonical space $\Omega := C_0(\mathbb R, \tilde{H})$ with the compact-open toplogy, $\mathcal F = \operatorname{Bor}(\Omega)$, and Wiener measure $\mathbb P$, so that $\mathbb P$-typical $\omega \in \Omega$ correspond to two-sided Brownian paths $\omega : \mathbb R \to \tilde{H}$ with $\omega(0) = 0$. We abuse notation somewhat and will write $W_t = W_t(\omega) = \omega(t)$ in the following. The space $(\Omega, \mathcal F, \mathbb P)$ will be equipped with the two-sided filtration $(\mathcal F_s^t)_{s < t}, \mathcal F_s^t := \omega(W_t - W_s)$, as well as the time-shift $\theta_t : \Omega \circlearrowleft$ given by $\theta^t(\omega)(s) := \omega(t + s)$, so that $(\Omega, \mathcal F, \mathbb P, (\theta_t)_{t \in \mathbb R})$ is an ergodic measure-preserving transformation (see, e.g., \cite{ChueshovScheutzow}). 
%{\color{red} [AB: Moved this up, got rid of ``$(\tilde \Omega, \tilde {\mathcal F}, \tilde {\mathbb P})$] }

	We will assume in what follows that $Q$ is given by 
	\begin{align} \label{eq:defineQ} Q:=(-\Delta)^{-2\beta} \, , \end{align}
	where $\beta \in (0,\frac{1}{4})$. While other choices are possible to make the forthcoming arguments work, this choice is made for ease of exposition. See Remark \ref{rmk:useOfCovariance} below for further discussion on precisely what is needed regarding the covariance $Q$.

The following is used to ensure propagation of $V_\gamma$-regularity, a crucial ingredient of the proofs in Section \ref{sec:proofs}. 
	\begin{lemma}\label{lem:stochConvolution}
		Let $\gamma \in (\frac14, \frac14 + \beta)$. With $Q$ as above, there exists an almost-sure, $\mathcal F_0^t$-adapted modification $(z_t)_{t \geq 0}$ of the stochastic convolution 
		\begin{align}
			\int_0^t S(t - s) \sqrt{Q} dW_s
		\end{align}
		for which $t \mapsto z_t$ is continuous with values in $V_\gamma$ for $t \geq 0$.
		% there exists a probability space $(\Omega, \mathcal F, \mathbb P)$, a two-sided filtration $(\mathcal F_s^t)_{s < t}$, and an $\mathcal F_0^t$-adapted Markov process $(z_t)$ on $(\Omega, \mathcal F, \mathbb P)$ serving as a pathwise mild solution of the linear SPDE	
	\end{lemma}
	\begin{proof}[Proof sketch]
		This follows from the same argument as that of \cite[Proposition 3.1]{ChueshovScheutzow}, using only the assumption that
		\begin{equation}\label{trace}
			\text{tr}_H(Q(-\cL)^{2\gamma-1+\varepsilon})<\infty,
		\end{equation}
		for some $\varepsilon > 0$, which is equivalent to the upper bound $\gamma < \beta + \frac14$. Here, $\text{tr}_H$ refers to the standard trace of an operator on $H$. 
%		It is straightforward to check that \eqref{trace} holds for any $\gamma \in (\frac14, \frac14 + \beta)$. 
	\end{proof}

	\subsubsection*{Well-posedness, $C^1$ semiflow and RDS formulation}
	 Let $\gamma \in (\frac14, \frac14 + \beta)$ be fixed for now. 
	By standard techniques, it now holds (see, e.g., \cite{DaPratoZabczyk}) that
	\begin{itemize}
		\item [1)] there exists $\bP$-a.s.~a unique mild solution of~\eqref{spde_intro} 
		\[
		u\in L^2(\Omega\times(0,T);V_\gamma) \cap 
%		L^4(\Omega\times(0,T)\times\cO)\cap 
		L^2(\Omega; C([0,T]; H))
		\]
		for all $T > 0$; and  %\textcolor{red}{[Adapt to $V_{\gamma}$?]}
		\item [2)] the first variation equation along the trajectory $(u_t)$, given by 
		\begin{align}\label{variational:eq}
		\txtd v = (\Delta v + \alpha v - 3 u^2 v) \txtd t 
		\end{align}
		is well-posed $\bP$-almost surely and for arbitrary initial data $v\in H$. 
%		\todo[inline]{We ought to give a little more detail on well-posedness for $v$. True, it's easier than the other one because the noise term drops out so you have classical solutions, but it's odd in view of the high level of detail in item (a). - AB}
	\end{itemize}

Below, we collect various properties that allow to realize mild solutions of \eqref{spde_intro} as the trajectories of a random $C^1$ semiflow $\varphi^t_\omega :  H \to H$.
	 The first three statements are well-known (\cite{CaraballoLangaRobinson}) and the Fr\'echet differentiablity follows from~\cite[Lemma 4.4]{Debussche}.
	\begin{proposition} \label{prop:RDS}
		There is a $\theta^t$-invariant subset $\Omega' \subset \Omega$ of full probability\footnote{In what follows, we will intentionally abuse notation and conflate $\Omega$ and $\Omega'$. 
		} such that for all $\omega \in \Omega'$ and any $t \geq 0$, 
		there is a (Fr\'echet differentiable) $C^1$ semiflow 
		\[
		u_0 \mapsto \varphi(t, \omega, u_0) =: \varphi^t_\omega(u_0)
		\]
		on $H = L^2(\cO)$ with the following properties for each $\omega \in \Omega'$: 
		\begin{itemize}
			\item[(a)] For all $T > 0$ and fixed initial $u_0 \in H$, the mapping $\Omega \times [0,T] \mapsto H$ given by $(\omega, t) \mapsto \varphi^t_\omega(u_0)$ is the (unique) pathwise mild solution to \eqref{spde_intro}. 
			\item[(b)] The semiflow $\varphi$ satisfies the \emph{cocycle property}: for $s, t > 0$ we have
			\[
			\varphi^{t + s}_\omega = \varphi^t_{\theta^s \omega} \circ \varphi^s_\omega
			\]
			\item[(c)] For any $u \in H$ and $s, t \in \bR, s < t$, we have that
			$\varphi^{t-s}_{\theta^s \omega}(u)$ is $\cF_s^t$-measurable as an $H$-valued random variable. 
			\item[(d)] For all $T > 0$ and fixed initial conditions $u_0, v_0 \in H$, the mapping
			$\Omega \times [0,T] \mapsto H$ given by $(\omega, t) \mapsto D_{u_0} \varphi^t_\omega (v_0)$ is the unique solution to the first variation equation \eqref{variational:eq}. 
		\end{itemize}
	\end{proposition}

 Here, given a $C^1$ Fr\'echet-differentiable 
	mapping $\psi : H \to H$, we write $D_u \psi \in L(H)$ for the derivative of $\psi$ evaluated at $u \in H$.

	%\todo[inline]{I don't known that we need everything here, but i thought it'd be helpful to write it down like this. In particular i'm not sure we need the ``perfected'' cocycle property as in item (d), but if it's not hard to show maybe a good thing to write down. Also, DPZ only proves Gateaux differentiability, we'd need to find a different place where Frechet is proved. Maybe not an important point, though, technically gateaux is all one needs to define finite-time Lyapunov exponents (asymptotic exponents require Frechet, I think...)}

	\subsubsection*{Markov process formulation}
	
	For fixed initial $u_0 \in H$ and for  $\omega \in \Omega$, 
	we will write $(u_t)_{t \geq 0}$ for the random process in $H$ defined by 
	$u_t := \varphi^t_\omega (u_0)$. We note two important properties of this process.
	\begin{lemma}
		The process $(u_t)$ is an $\cF^t_0$-adapted, Feller Markov process. 
	\end{lemma}
	\begin{proof}
		That $(u_t)$ is $\cF^t_0$-adapted follows from its definition and Proposition~\ref{prop:RDS}(c). The fact that it is a Feller Markov process follows from continuity of $u_0 \mapsto \varphi^t_\omega(u_0)$ for almost all $\omega \in \Omega'$.
	\end{proof}
	
	Furthermore, we obtain the following statement regarding the existence, uniqueness and regularity of the invariant measure associated to~\eqref{spde_intro}.

	\begin{proposition} \label{prop:statmeas}
		The process $(u_t)$ admits a unique, locally positive stationary measure $\rho$ on $H$ for which $\rho(V_\gamma) = 1$. 
	\end{proposition}
	
%	{\color{red} [AB: This is where we should state local positivity. Do we have it? ] } \textcolor{blue}{yes, the definition of the support should be enough, we know that every open neighbourhood of each point of the support has positive probability}

	\begin{proof}
	The existence of an invariant measure is an application of~\cite[Theorem 4.4]{stannat}, whereas its uniqueness can be inferred from~\cite[Section 5]{cerrai:ergodicity}. The hypothesis in~\cite{cerrai:ergodicity} can be verified in our setting, due to the structure of the covariance operator of the noise $Q=(-\Delta)^{-2\beta}$ for $\beta\in(0,\frac{1}{4})
$.	More precisely letting $(\lambda_k)_{k\in\mathbb{N}}$ denote the eigenvalues of the Dirichlet-Laplacian and $(q_k)_{k\in\mathbb{N}}$ stand for the eigenvalues of $Q$, the condition
\[ \sum\limits_{k=1}^\infty \frac{q^2_k}{\lambda^{1-p}_k}<\infty  \]
holds for some $p\in(0,1)$. This reduces to \[ \sum\limits_{k=1}^\infty \frac{1}{k^{4\beta+2-2\gamma}}<\infty,\]
which is satisfied provided that $\beta>\frac{p}{2}-\frac{1}{4}$. Therefore it is always possible to find $p\in(0,1)$ satisfying the above inequality. 
 That the resulting invariant measure fully charges $V_\gamma$ follows by construction relying on the Krylov-Bogoliubov method~\cite[Section 2]{stannat} and regarding the compact embeddings $V_\gamma\hookrightarrow V_{\gamma'}$  for $\gamma'\leq \gamma $. 
 
 Lastly, this measure is locally positive on $H$ by \cite[Proposition 8.3.6]{cerrai:ergodicity}. Since $\rho(V_\gamma) = 1$, it is straightforward to check that local positivity on $V_\gamma$ also holds. 
	\end{proof}
%\begin{remark}
%Based on the results in~\cite{cerrai:ergodicity} and~\cite[Section 8.3]{cerrai}, the measure $\rho$ is also concentrated on the space of continuous functions $C(\cO)$. \textcolor{red}{If this measure is locally positive in $C(\cO)$, the statements obtained in Appendix for $V_\gamma$ carry over to $C(\cO)$ as well.}
%\end{remark}

	%\begin{itemize}
	%Section 6.4 in \cite{DaPratoDebusscheGoldys}: The invariant measure $\rho$ for~\eqref{spde1} has a density w.r.t a Gausssian measure in $H$ which is given by
	%    \begin{align} \label{eq:density}
	%        p(x)=K e^{-\frac{1}{2}U(x)},
	%    \end{align}
	%    where $K$ is a normalization constant and $U$ is the antiderivative of $\alpha u -u^3$, i.e.~$U(x) = \frac{\alpha x^2}{2} - \frac{ x^4}{4}$.
	%\end{itemize}
	
	\subsection{Properties of attractor and sample measures}
	
	Given a Borel-measurable mapping $\psi : H \to H$ and a Borel probability 
	$\mu$ on $H$, define
	$\psi_* \mu := \mu \circ \psi^{-1}$ to be the pushforward of $\mu$ by $\psi$. 
	We recall a, by now, classical result from the RDS literature, sometimes also coined the \emph{correspondence theorem} between stationary measures $\rho$ and sample measures $\rho_{\omega}$ that are measurable with respect to the past, also named \emph{Markov measures}.
	\begin{lemma}[Theorem 4.2.9 of \cite{KuksinShiri}] \label{lem:samples}
		For $\bP$ a.e. $\omega \in \Omega'$, the weak$^*$ limit
		\[
		\rho_\omega = \lim_{t \to \infty} (\varphi^t_{\theta^{-t} \omega})_* \rho
		\]
		exists and is $\cF_- := \cF_{-\infty}^0$-measurable. The sample measures
		$\rho_\omega$ satisfy $(\varphi^t_\omega)_* \rho_\omega = \rho_{\theta^t \omega}$ with probability 1 for all $t \geq 0$ as well as $\bE(\rho_\omega) = \rho$. 
	\end{lemma}

	The sample measures can be associated with a unique attracting random equilibrium for the situation of~\eqref{spde_intro}, by combining Proposition~\ref{prop:statmeas}, Lemma~\ref{lem:samples}, the almost sure order-preservation $\varphi_{\omega}^t u \leq \varphi_{\omega}^t v$ for all $u \leq v$~\cite[Theorem 5.8]{CV} and the characterization of random attractors by Arnold and Chueshov \cite{ArnoldChueshov}.
	\begin{proposition}\label{prop:attractor_point}
		Consider the RDS induced by the SPDE~\eqref{spde_intro} for any value of $\alpha \in \mathbb{R}$. We have that
		\begin{enumerate}[(a)]
			\item with probability 1, the sample measure $\rho_\omega$ is atomic, i.e., 
			$\rho_\omega = \delta_{a(\omega)}$ where $a = a_\alpha : \Omega \to H$ is an 
			$\cF_-$-measurable $H$-valued random variable such that almost surely
			\[
			\varphi^t_\omega(a(\omega)) = a(\theta^t(\omega)),
			\]
			\item the set valued map $ \omega \mapsto\{a(\omega)\}$ is the unique random attractor of the RDS induced by the SPDE~\eqref{spde_intro}, i.e.~for all bounded $D \subset H$ we have almost surely
			$$ \lim_{t \to \infty}  \sup_{d \in D} \|\varphi_{\theta_{-t}\omega}^t (d) - a(\omega)\|_{H} =0.$$
		\end{enumerate}
	\end{proposition}
	\begin{proof}
While the covariance operator $Q$ is slightly different in our case, the main arguments from~\cite[Theorem 6.1]{Caraballoetal} carry over. The existence of the attractor follows upon reducing the SPDE~\eqref{spde_intro} into a PDE with random coefficients, which utilizes the stochastic convolution from Lemma \ref{lem:stochConvolution}. Then, one can easily derive an absorbing set using a-priori estimates  and the fact that $F:V_\gamma\to V_\gamma$ is locally Lipschitz. The compactness of the absorbing set follows from the compact embeddings $V_\gamma\hookrightarrow V_{\gamma'}$ for $ \gamma'\leq \gamma$.
	The fact that the random attractor is a singleton is implied by the order-preservation of the system together with the existence of a unique invariant measure, recalling Proposition~\ref{prop:statmeas}. 
	\end{proof}	
 For proving the bifurcations in terms of finite-time Lyapunov exponents, we crucially require the following lemma on 
	regularity of the random attractor $a = a_\alpha$. 
	\begin{proposition}\label{h1:attractor} \
		\begin{itemize}
			\item[(a)] With probability 1, we have that $a(\theta^t \omega) \in V_\gamma$ for all $t \in \bR$. 
			\item[(b)] For any $T > 0$, there exists a $\mathcal{F}_{-\infty}^T$-measurable set $\cA \subset \Omega$ with $\bP(\cA) > 0$ such that
			\begin{align*}
			\|a_{\alpha}(\theta^s\omega)\|_{V_\gamma}\in(0,\varepsilon)\quad \text{ for all } s\in[0,T] \text{ and }\omega\in \cA.
			\end{align*} 
		\end{itemize}
	\end{proposition}

	Roughly, Proposition \ref{h1:attractor} will follow from (i) the fact that $\rho(V_\gamma) = 1$, hence $a_\alpha(\omega) \in V_\gamma$ with probability 1 and (ii) that $V_\gamma$-regularity is propagated\footnote{Note that this regularity issue is inherent to the infinite-dimensional problem and does not arise in the previous SODE works \cite[Proposition 4.1]{Callawayetal} and~\cite[Proposition 5.1]{DoanEngeletal}.} in time by the evolution equation \eqref{spde_intro}, and that $\| a(\theta^t \omega)\|_{V_\gamma}$ can be made small by taking $\| a(\omega)\|_{V_\gamma}$ small. Point (i) is immediate from the preceding discussion in this section, while point (ii) is ensured by the existence of the $V_\gamma$-continuous modification of the stochastic convolution in Lemma \ref{lem:stochConvolution} together with regularizing properties of analytic semigroups. See Appendix \ref{appendix} for further details.

%Note that this is a first important technical challenge compared to the finite-dimensional case, where the analogue of statement (b) is immediate from the fact that the stationary measure is supported on the entire real line, c.f.~\cite[Proposition 4.1]{Callawayetal} and~\cite[Proposition 5.1]{DoanEngeletal}. In the infinite-dimensional setting the regularity of the attractor follows from the fact that it is distributed according to a unique invariant measure on $V_\gamma$, whereas (b) holds true since this measure is fully supported on $V_\gamma$, as stated in Proposition~\ref{prop:statmeas}.
	\begin{remark}
	We refer the reader to~\cite{Chueshov2000,Zhao} for further details regarding the regularity of random attractors for stochastic reaction diffusion equations with finite-dimensional additive noise, based on a random dynamical systems approach (without using the correspondence between attractors and invariant measures).	
		 	\end{remark}
	%Our main result, Theorem A, concerns bounds on FTLEs, as introduced by~\eqref{eq:Lambdai}. Note for the following that we have, in particular, that
	%$$ \Lambda_1(t, a(\omega)) = \frac{1}{t} \ln \| D_{a_\alpha(\omega)} \varphi^t_\omega \|.$$
	
	%\todo[inline]{Definitely correct also in infinite dimensions? Maybe more on FTLEs here, also on relation to asymptotic LEs. Would it additionally make sense to introduce a finite-time Kolmogorov Sinai entropy to compare to the sum of positive Lyapunov exponents?}

	\section{Proofs of Theorems \ref{thm:topLEEst} and \ref{thm:volumeFTLE}} \label{sec:proofs}
	
	{\bf Notation}: In the following, the vectors $e_1, e_2, \dots$ denote the orthonormal Fourier basis of the Laplacian $\Delta$ on $\cO = [0,L]$ with Dirichlet boundary conditions, 
	\[
	e_k(x) = \sqrt{\frac{2}{L}}\sin(2 \pi k x / L).
	\]
	%up to normalization. 
	We write $\pi_1$ for the orthogonal projection onto 
	the span of $e_1$, and $( \cdot, \cdot) = ( \cdot, \cdot)_{L^2}$ for the $L^2$ inner product. Moreover, we write $\lambda_k = (2 \pi k / L)^2$ for the corresponding eigenvalues of $(-\Delta)$, so that $\Delta e_k = - \lambda_k e_k$. Lastly, in this section we will write $\| \cdot \| = \| \cdot\|_H$ for clarity of notation. 
	
	\subsection{Proof of Theorem \ref{thm:topLEEst}: estimate of top FTLE}

	The following summarizes our estimates of finite-time Lyapunov exponents along the attracting random equilibrium $a_{\alpha}(\omega)$ and immediately implies Theorem \ref{thm:topLEEst}. 
	\begin{proposition} \label{prop:FTLE_1} \
		\begin{itemize} 
			\item[(a)] Unconditionally, 
			\[
			\| D_{a_\alpha(\omega)} \varphi^t_\omega  \| \leq e^{t ( \alpha - \lambda_1)}
			\]
			for all $t > 0$ and with probability 1. 
			\item[(b)] Assume $\alpha - \lambda_1 > 0$. For all $0 < \eta \ll \alpha - \lambda_1, T > 0$ there exists $\cA \in \cF^T_{-\infty}$ with $\mathbb P(\cA) > 0$ such that 
			\[
			\| D_{a_\alpha(\omega)} \varphi^t_\omega \| \geq (1 - \eta) e^{t (\alpha - \lambda_1 - \eta)}
			\]
			for all $t \in [0,T]$, $\omega \in \cA$. 
		\end{itemize}
	\end{proposition}
	
	\subsubsection*{Proof of the upper bound (a)}
	The bound from above in part (a) is straightforward from the energy estimate derived by taking a time derivative of $\| v_t\|^2$, where 
	$v_t := D_{a(\omega)} \varphi^t_\omega v_0$ for fixed $v_0 \in L^2$. Recalling with \eqref{variational:eq} the variational random PDE
	\begin{equation} \label{eq:var_eq_attract}
	\dot v_t = (\alpha + \Delta) v_t - 3 a_\alpha(\theta_t\omega)^2 v_t \, ,
	\end{equation}
	we see that
	\[
	\frac12 \frac{\txtd}{\txtd t} \| v_t\|^2 = ( v_t, \dot v_t) = (v_t, (\alpha + \Delta) v_t) - (v_t, 3 a_\alpha(\omega)^2 v_t) \leq (v_t, (\alpha + \Delta) v_t) \, .
	\]
	Since $(\alpha + \Delta)$ is self-adjoint with eigenvalues $\alpha - \lambda_i, i \geq 1$, we see in view of the min-max principle for closed self-adjoint operators that 
	\[
	(w, (\alpha + \Delta) w) \leq (\alpha - \lambda_1) \| w\|^2
	\]
	for all $w \in L^2$. In conclusion, 
	\[
	\frac{\txtd}{\txtd t} \log \| v_t\| \leq \alpha - \lambda_1 \, ,
	\]
	from which the estimate in (a) follows. 
	
	\subsubsection*{Proof of the lower bound (b)}
	Our primary tool in this proof will be the family of quadratic forms on $L^2 \times L^2$
	\[
	Q_\delta(v, w) = \delta (\pi_1 v, \pi_1 w)_{L^2} -   (\pi_1^\perp v, \pi_1^\perp w)_{L^2} \, , \quad \delta > 0 \, . 
	\]
	In what follows, we will abuse notation and write $Q_\delta(v) = Q_\delta(v, v)$. Conceptually, quadratic forms such as these specify closed cones
	% I inserted $\geq 0$.
	\[
	\cC_\delta = \{ v \in L^2 : Q_\delta(v) \geq 0 \} = \left\{ \| \pi_1^\perp v\|^2 \leq \delta \| \pi_1 v\|^2 \right\} \, , \quad \delta > 0 \, , 
	\]
	roughly parallel to the span of the first eigenmode $e_1$. 
	It is evident that the shifted heat semigroup $e^{(\alpha + \Delta) t}$ leaves
	these cones $\cC_\delta$ invariant, expanding vectors within them to order
	$e^{t (\alpha - \lambda_1)}$.  
	
	In summary, our lower bound on $\| D_{a(\omega)} \varphi^t_\omega\|$ will come from showing the following:  (1) the operator $D_{a(\omega)} \varphi^t_\omega$ is close to the heat semigroup $e^{(\alpha + \Delta) t}$ conditioned on an event $\cA$ along which the perturbation factor  $3 a_\alpha(\theta_t \omega)^2$ in \eqref{eq:var_eq_attract} is small in an appropriate sense; and (2) we can transfer cone preservation and vector expansion of the heat semigroup to the nearby time-$t$ cocycle $D_{a(\omega)} \varphi^t_\omega$. 
	The following makes this more precise: 
	\begin{lemma} \label{lem:topFTLE}
		Assume $\alpha > \lambda_1$. Let $T > 0$ and $\varepsilon > 0$ with $ \sqrt{\varepsilon} \ll \alpha - \lambda_1$ be fixed, and let $\omega \in \Omega$ be a noise path with the property that the nonlinear term $B_\omega^t := -3 a_\alpha(\theta^t \omega)^2$ satisfies 
		\begin{align}\label{eq:makeNemytskiiSmall}
		\| B_\omega^t\|_{\cC(\cO)} \leq \varepsilon
		\end{align}
		for all $t \in [0,T]$. 
		%Finally, let $\delta > 0$ be such that
		%\begin{align}\label{eq:conditionOnDelta}
		%(1 + \delta) \varepsilon + \frac{(1 + \delta) \varepsilon}{\delta} \leq \alpha - \lambda_1
		%\end{align}
		Finally, assume $v_0 \in L^2$ satisfies $Q_{\delta}(v_0) > 0$, where 
		\[
		\delta := \sqrt{\varepsilon}.
		\] 
		
		Under these conditions, the time-$t$ solution $v_t = D_{a(\omega)} \varphi^t_\omega(v_0)$ to the first variation equation \eqref{eq:var_eq_attract} satisfies
		\begin{align}\label{eq:lowerBoundQdelt1}
		\frac12 \frac{\txtd}{\txtd t} Q_{\delta}(v_t) \geq \left( \alpha - \lambda_1 - %\frac{(1 + \delta) \varepsilon}{\delta}
		2 \delta
		\right) Q_{\delta}(v_t) \, .
		\end{align}
	\end{lemma}
	While in Section~\ref{subsec:bounding_FTLE} we will prove a more general result, we have included the following proof for convenience of the reader. 
	\begin{proof} 
To start, observe that 
		\begin{align*}
	\frac12 \frac{\txtd}{\txtd t} Q_\delta(v_t) & = Q_\delta(v_t, \dot v_t)  = \delta (\pi_1 v_t, \pi_1 \dot v_t) - (\pi_1^\perp v_t, \pi_1^\perp \dot v_t) \,,  
	\end{align*}
hence
		\begin{align*}
	\frac{1}{2} \frac{\txtd}{\txtd t} Q_\delta(v_t) & 
	\geq \delta (\alpha - \lambda_1) \| \pi_1 v_t\|^2 - (\alpha - \lambda_2) \| \pi_1^\perp v_t\|^2 - (1 + \delta) \varepsilon \| v_t\|^2 \\
	& \geq  \left(\alpha - \lambda_1 - \frac{(1 + \delta)\varepsilon}{\delta} \right) \delta \| \pi_1 v_t\|^2 - \left( \alpha - \lambda_2 + (1 + \delta) \varepsilon \right) \| \pi_1^\perp v_t\|^2 \, , 
	\end{align*}
	having used the estimates
	\begin{gather*}
	( \pi_1 v , (\alpha + \Delta) \pi_1 v) \geq (\alpha - \lambda_1) \| \pi_1 v\|^2 \, ,  \\
	( \pi_1^\perp v , (\alpha + \Delta) \pi_1^\perp v) \leq (\alpha - \lambda_2) \| \pi^\perp_1 v\|^2
	\end{gather*}
	and decomposing $\| v\|^2 = \| \pi_1 v \|^2 + \| \pi_1^\perp v\|^2$. 
	On assuming that 
	\[
		(1 + \delta) \varepsilon \leq \alpha - \lambda_1 - \frac{(1 + \delta) \varepsilon}{\delta} \,, 
	\]
	which can be arranged with $\delta = \sqrt{\varepsilon}$ and $\varepsilon$ taken sufficiently small, 
	it follows that 
	\begin{align*}
	\frac12 \frac{\txtd}{\txtd t} Q_\delta(v_t) \geq \left( \alpha - \lambda_1 - \frac{(1 + \delta) \varepsilon}{\delta} \right) Q_\delta(v_t) \, , 
	\end{align*}
	which implies the desired bound. 
%	The statement follows assuming that
%	\[
%	(2 + \delta) \varepsilon \leq \alpha - \lambda_1 - \frac{(2 + \delta) \varepsilon}{\delta},
%	\]
%	which is satisfied for e.g. $\delta = \sqrt{\varepsilon}$, where $\varepsilon > 0$ is taken small enough in terms of $\alpha - \lambda_1$. 
	\end{proof}
	
	\begin{proof}[Completing the proof of Proposition \ref{prop:FTLE_1}(b)]

%		Let $\mathcal A \in \mathcal F_{-\infty}^T$ be as in Lemma \ref{h1:attractor}. 	

		Assume for the moment that equation \eqref{eq:lowerBoundQdelt1} holds for all $t \in [0,T]$. Then,  
	\begin{equation*}
	\frac12 \frac{\frac{\txtd}{\txtd t} Q_\delta(v_t)}{Q_\delta(v_t)} \geq  \alpha - \lambda_1 - 2 \delta  \, , 
	\end{equation*}
hence
,	\begin{equation} \label{est:Qdeltat}
	Q_\delta(v_t)  \geq Q_\delta(v_0) \exp \{2t \left( \alpha - \lambda_1 - 2 \delta \right)\} \,. 
	\end{equation}
		To translate this to a lower bound on norms: Take $M > 1$ and assume $Q_{\delta / M}(v_0) \geq 0$. Then, 
	\[
	Q_{\delta}(v_0) = \delta \| \pi_1 v_0 \|^2 - \| \pi_1^\perp v_0 \|^2
	= \delta \left( 1 - \frac{1}{M} \right) \| \pi_1 v_0 \|^2 + Q_{\delta / M}(v_0) \geq \frac{\delta(M-1)}{M} \| \pi_1 v_0 \|^2. %\geq \frac{\delta }{2 + \delta} \| v_0\|^2 \, .
	\]
	Using that $v \in \cC_\eta$ implies $\| v\|^2 \leq (1 + \eta) \| \pi_1 v\|^2$ for $\eta > 0$, we have
	\begin{equation} \label{est:Qdelta0}
	Q_\delta(v_0) \geq \frac{\delta (M-1)/M}{1 + \delta / M} \| v_0\|^2
	= \frac{\delta (M-1)}{M + \delta} \| v_0\|^2 \, .
	\end{equation}
	Since $Q_\delta(w) \leq \delta \| w\|^2$ unconditionally for all $w \in L^2$, we combine~\eqref{est:Qdeltat} and~\eqref{est:Qdelta0} to obtain
	\[
	\| v_t\|^2 \geq \frac{M-1}{M + \delta} \exp \{2t \left( \alpha - \lambda_1 - 2 \delta \right) \}\| v_0 \|^2 \quad \text{ for all } \quad v_0 \in \cC_{\delta / M} \, .
	\]
	We can clearly make the prefactor as close to 1 as desired on taking $M$ sufficiently large.

To complete the proof, it suffices to arrange for \eqref{eq:lowerBoundQdelt1} to hold for all $t \in [0,T]$ on a positive-probability event $\mathcal A \in \mathcal F_{-\infty}^T$. This follows by Lemma \ref{lem:topFTLE} as long as $\| B^t_\omega\|_{L^\infty}$ can be made small as in \eqref{eq:makeNemytskiiSmall} for all $t \in [0,T]$. This, in turn, is implied by Proposition \ref{h1:attractor}, which allows to make $a_\alpha(\theta^t \omega)$ small in $V_\gamma$, hence also in $\cC(\cO)$, for all $t \in [0,T]$.  
	\end{proof}

	 \begin{remark}\label{rmk:useOfCovariance} 
Note that the only information we used in the preceding argument was that $a_\alpha(\theta^t\omega)$ could be made sufficiently small in $V_\gamma$ (hence also in $
\cC(\mathcal O)$) for all $t \in [0,T]$. Indeed, the argument goes through with $a_\alpha(\theta^t \omega)$ replaced by any trajectory $(u_t)_{t \geq 0}$, as long as $\| u_t\|_{V_\gamma}$ remains sufficiently small for all $t \in [0,T]$. By Lemma \ref{lem:propagationReg} in the Appendix, this can be arranged with positive probability for all fixed initial $u_0$ with $\| u_0\|_{V_\gamma}$ sufficiently small. Indeed, this alternative version of 
 Proposition \ref{prop:FTLE_1} makes no use at all of the stationary measure or the point attractor, and relies only on the RDS framework of Proposition \ref{prop:RDS}, hence only on the assumption \eqref{trace} on the covariance operator $Q$. These observations apply equally well to the forthcoming arguments of Section \ref{subsec:bounding_FTLE} controlling finite-time $k$-dimensional volumes. 
	 \end{remark}

	\subsection{Proof of Theorem \ref{thm:volumeFTLE}: bounding volume growth} \label{subsec:bounding_FTLE}
	%\todo[inline]{Explain meaning of singular values and compare to asymptotic LEs, here or in Intro. Have we confirmed compactness of our operator, as referred to below?}
	
	We now wish to apply similar ideas to estimate volume growth rates. We begin with some background on wedge spaces and norms in Section \ref{subsubsec:wedges}, allowing us to state Proposition \ref{prop:FTvol}  which summarizes the volume growth bounds required in Theorem \ref{thm:volumeFTLE}. 
	The proof of Proposition \ref{prop:FTvol} via a cones argument will occupy the remainder of the Section.

	%regarding estimates  Proposition \ref{prop:FTLE_1} for volumes is stated and proved in 
	
	%We wish now to apply similar ideas to estimate `lower' FTLE, namely, 
	%the finite-time exponential growth rates of singular values. Recall that
	%for a linear operator $A$ on a Hilbert space $H$, the 
	%singular values $\sigma_i(A)$ are the eigenvalues of $A^* A$ 
	%written in decreasing order, counted with multiplicity: 
	%$\sigma_1(A) \geq \sigma_2(A) \geq \cdots$. When $A$ is a compact linear operator, 
	%note that $\sigma_i(A) \to 0$ as $i \to \infty$. 
	%
	%\begin{proposition} \label{prop:singvalues}
	%\begin{itemize}
	%	\item[(a)] Unconditionally, for all $i = 1,2, \dots$,
	%	\[
	%	\sigma_i(D_{a(\omega)} \varphi^t_\omega) \leq e^{t ( \alpha - \lambda_i)}
	%	\]
	%	for all $t > 0$ with probability 1. 
	%	
	%	\item[(b)] Assume $\alpha - \lambda_r > 0 > \alpha - \lambda_{r + 1}$ for some $r \geq 1$. Let $\eta > 0$ be sufficiently small, $T > 0$ arbitrary. Then, 
	%	there exists $\cA \in \cF^T_{-\infty}$ with $\mathbb P(\cA) > 0$ such that for all $\omega \in \cA, t \in [0,T]$, we have
	%	\[
	%	\sigma_i(D_{a(\omega)} \varphi^t_\omega) \geq (1 - \eta) e^{t (\alpha - \lambda_i - \eta)} 
	%	\]
	%	for each $i \in \{ 1, \cdots, r\}$. 
	%\end{itemize}
	%\end{proposition}
	
	%Our approach for the rest of this section, which will yield the proof of Proposition~\ref{prop:singvalues}, will be more formal, and along the way we will derive the proof of Lemma \ref{lem:topFTLE} above used to estimate the top FTLE. 
	
	\subsubsection{Background on wedge spaces}\label{subsubsec:wedges}
	
	Let $H$ be a separable Hilbert space. Given $v_1, \cdots, v_k \in H$ 
	we write 
	\[
	v_1 \wedge \cdots \wedge v_k
	\]
	for the \emph{wedge product} of $\{ v_i\}$ (sometimes referred to as a $k$-blade), and write $\wedge^k H$
	for the closure of the set of finite linear combinations of 
	$k$-blades under the norm induced by the inner product
	\[
	( v_1 \wedge \cdots \wedge v_k, w_1 \wedge \cdots \wedge w_k) 
	= \det \left[ (v_i, w_j)_{ij} \right] \, , 
	\]
	where $(v_i, w_j)_{ij}$ denotes the $k \times k$ matrix with $i,j$-th entry 
	$(v_i, w_j)$. Let $\| \cdot\|_{\wedge^k H}$ denote the norm induced by this inner product. 
	
	We recall the following elementary properties of $\wedge^k H$: 
	\begin{itemize}
		\item[(1)] If $e_1, e_2, \cdots \in H $ is a complete orthonormal system, then 
		an orthonormal basis for $\wedge^k H$ is given by the set of $k$-blades
		\[
		e_{\bf i} =  e_{i_1} \wedge \cdots \wedge e_{i_k},
		\]
		as ${\bf i} = (i_1, \cdots, i_k)$ ranges over the set of all distinct indices
		$i_1 < i_2 < \cdots < i_k$. 
		
		\item[(2)] A bounded linear operator $A$ on $H$ gives rise to an operator
		$\wedge^k A$ in $\mathcal B(\wedge^k H)$, i.e.~the space of bounded linear operators from $\wedge^k H$ into itself, via 
		$$\wedge^k A (v_1 \wedge \cdots \wedge v_k) = A v_1 \wedge \cdots \wedge A v_k.$$
		This operator has the property that
		\begin{align}\label{eq:normOfWedgeOperator}\begin{aligned}
		\| \wedge^k A\|_{\wedge^k H} &= \sup\{ | \det(A|_E) | : E \subset H , \dim E = k \} \\
	& = \sup\left\{ {\| A v_1 \wedge \dots \wedge A v_k \|_{\wedge^k H} \over \| v_1 \wedge \dots \wedge v_k \|_{\wedge^k H}} : v_1, \dots, v_k \in H \text{ linearly independent}\right\}
		\end{aligned} \end{align}
		where for the purposes of defining $\det$ we view $A|_E : E \to A(E)$ as a linear operator of finite-dimensional inner product spaces, and set $\det(A|_E) = 0$ if $\dim A(E) < \dim E$. 
		%Moreover, this operator has the property that
		%\[
		%\sigma_1(A) \cdots \sigma_k(A) = \| \wedge^k A\|_{\wedge^k H} \, .
		%\]
	\end{itemize}
	Below, we will abuse notation somewhat and write $\| \cdot\| = \| \cdot\|_{\wedge^k H}$. 
	
	We are now in position to formulate the estimates needed in the proof of Theorem \ref{thm:volumeFTLE}. 
	\begin{proposition}\label{prop:FTvol} \
		\begin{itemize}
			\item[(a)] Unconditionally, for all $k \geq 1$ we have that
			\[
			\|\wedge^k D_{a(\omega)} \varphi^t_\omega\| \leq e^{t \sum_{i =1 }^k ( \alpha - \lambda_i)}
			\]
			for all $t > 0$ with probability 1. 
			
			\item[(b)] Assume $\alpha - \lambda_r > 0 > \alpha - \lambda_{r + 1}$ for some $r \geq 1$. Let $\eta > 0$ be sufficiently small, $T > 0$ arbitrary. Then, 
			there exists $\cA \in \cF^T_{-\infty}$ with $\mathbb P(\cA) > 0$ such that for all $\omega \in \cA, t \in [0,T]$, we have
			\[
			\| \wedge^k D_{a(\omega)} \varphi^t_\omega \| \geq (1 - \eta) e^{t \sum_{i = 1}^k (\alpha - \lambda_i - \eta)} 
			\]
			for each $k \in \{ 1, \cdots, r\}$. 
		\end{itemize}
		
	\end{proposition}

	\subsubsection{Proof of Proposition \ref{prop:FTvol}(a): upper bound on $k$-dimensional volume growth}
	
	In using wedge products to derive upper bounds we follow a long tradition of authors, e.g., Temam~\cite{Temam}, Debussche~\cite{Debussche}). 
	
	To this end, let $k \geq 1$ and fix ${\bf v}_0 = v^1_0 \wedge \cdots \wedge v^k_0 \in \wedge^k H$, writing
	\[
	{\bf v}_t = v^1_t \wedge \cdots \wedge v^k_t = \wedge^k D_{a(\omega)} \varphi^t_\omega ({\bf v}_0) \,. 
	\] 
	Recalling that $\partial_t v_t^i  = (\alpha + \Delta + B_\omega^t) v_t^i$ for all $i=1, \dots, k$, we compute
	\begin{align}\label{eq:energyEstimateUpperBoundWedgeK}
\frac12 \frac{\txtd}{\txtd t} \| {\bf v}_t\|^2 
	= \sum_{j = 1}^k ( {\bf v}_t , v_t^1 \wedge \cdots \wedge (\alpha + \Delta + B_\omega^t ) v_t^j \wedge \cdots \wedge v_t^k ).
	\end{align}
	We will use the following linear algebra lemma. 
	\begin{lemma}
		Let $B$ be a bounded, negative semi-definite operator on a separable Hilbert space $H$ and let $k \geq 1$. Then, the operator $\hat B^{(k)}$ on $\wedge^k H$
		defined by
		\[
		\hat B^{(k)} (v_1 \wedge \cdots \wedge v_k) = \sum_{j = 1}^k v_1 \wedge \cdots \wedge B v_j \wedge \cdots \wedge v_k 
		\]
		is negative semi-definite as an operator on $\wedge^k H$. 
	\end{lemma}
	\begin{proof}
		To start, observe that
		\[
		\hat B^{(k)} = \frac{\txtd}{\txtd t} \bigg|_{t = 0} \wedge^k e^{B t} 
		\]
		where on the RHS is the Frechet derivative at $t = 0$ of $t \mapsto \wedge^k e^{B t}$. Immediately it follows that $\hat B^{(k)}$ is self-adjoint. To check it is negative semidefinite, we will show that $\wedge^{k} e^{B t}$ is a contraction semigroup, i.e., 
		\[
		\| \wedge^k e^{B t}\| \leq 1
		\]
		for all $t > 0$. If this is the case, then
		\[
		( \hat B^{(k)}(v_1 \wedge \cdots \wedge v_k) , v_1 \wedge \cdots \wedge v_k) = \frac12 \frac{\txtd}{\txtd t} \bigg|_{t = 0} \| \wedge^k e^{B t}(v_1 \wedge \cdots \wedge v_k)\|^2 \leq 0
		\]
		for all $v_1 \wedge \cdots \wedge v_k \in \Lambda^k H$, hence $\hat B^{(k)}$ is negative semidefinite.

		To check that $\wedge^k e^{B t}$ is a contraction semigroup, it suffices by the characterization in \eqref{eq:normOfWedgeOperator} to estimate $\| \wedge^k e^{B t} (v_1 \wedge \dots \wedge v_k)\|$ for some linearly independent set $\{ v_1, \cdots, v_k\} \subset H$. Applying the Gram Schmidt process to this set of vectors, let $\{ w_1, \cdots, w_k\} \subset H$ be an orthogonal set for which
		$w_i \in \operatorname{Span}\{ v_1, \dots, v_i\}$ for each $1 \leq i \leq k$. On cancelling repeated wedge terms of the form $v_i \wedge v_i$, it follows that
 		\[
		v_1 \wedge \cdots \wedge v_k = w_1 \wedge \cdots \wedge w_k \, .
		\]
		Then,
		\begin{align*}
		\| \wedge^k e^{B t} (v_1 \wedge \cdots \wedge v_k) \| & = 
		\| \wedge^k e^{B t} (w_1 \wedge \cdots \wedge w_k)\| \\
		& \leq \prod_{i =1 }^k \| e^{B t} w_i\| \leq \prod_{i =1 }^k \| w_i\| \, ,
		\end{align*}
		having used that $B$ negative semi-definite implies that $e^{B t}$ is a contraction semigroup. Note now that since the $\{ w_i\}$ are orthogonal, $(w_i, w_j) = \| w_i\|^2 \delta_{ij}$ (here $\delta_{ij} = 1$ if $i = j$ and $= 0$ if $i \neq j$), hence $\det (w_i, w_j) = \prod_{i = 1}^k \| w_i\|^2$. In view of the definition of $\| \cdot \| = \| \cdot \|_{\wedge^k H}$, we conclude that
		\[
		\prod_i \| w_i \| = \| w_1 \wedge \cdots \wedge w_k \| = \| v_1 \wedge \cdots \wedge v_k\| \, , \]
		completing the proof. 
	\end{proof}
	
	To complete the upper bound as in part (a) of Proposition~\ref{prop:FTvol}, observe from 
	\eqref{eq:energyEstimateUpperBoundWedgeK} and the previous Lemma that
	\[
	\frac12 \frac{\txtd}{\txtd t} \| {\bf v}_t\|^2 \leq \sum_{j = 1}^k ({\bf v}_t, v_t^1 \wedge \cdots \wedge (\alpha + \Delta) v_t^j \wedge \cdots \wedge v_t^k) 
	\leq \left( \sum_{j = 1}^k (\alpha - \lambda_j) \right) \| {\bf v}_t\|^2
	\]
	by the min-max principle. The estimate in (a) now follows. 
	
	\subsubsection{Proof of Proposition \ref{prop:FTvol}(b): Quadratic forms on $\wedge^k H$}
	
	\newcommand{\bi}{{\bf i}}
	\newcommand{\ve}{\varepsilon}
	
	For $\delta > 0$ 
	we define the quadratic form $Q^{(k)}_\delta$ on $\wedge^k H$ 
	\begin{align*}
	Q_\delta^{(k)} (v_1 \wedge \cdots \wedge v_k , w_1 \wedge \cdots \wedge w_k) &= \delta \langle \wedge^k \Pi_k (v_1\wedge \cdots \wedge v_k) , 
	w_1 \wedge \cdots \wedge w_k \rangle  \\
	& -  \langle \wedge^k \Pi_k^\perp (v_1\wedge \cdots \wedge v_k) , 
	w_1 \wedge \cdots \wedge w_k \rangle
	\end{align*}
	where $\Pi_k$ denotes orthogonal projection onto the span of the first $k$ eigenmodes $\{ e_1, \cdots, e_k\}$ of $\Delta$, and $\Pi_k^\perp = I - \Pi_k$. 
	Equivalently, $\wedge^k \Pi_k$ is the orthogonal projection onto the span of $e_{\bi_0} = e_1 \wedge \cdots \wedge e_k, \bi_0 := (1,\cdots, k)$. In what follows, we will again abuse notation and write $$Q_\delta^{(k)} (v_1 \wedge \cdots \wedge v_k ) = Q_\delta^{(k)} (v_1 \wedge \cdots \wedge v_k ,v_1 \wedge \cdots \wedge v_k).$$
	
	The following elaboration on Lemma \ref{lem:topFTLE} above extends that result to 
	the operator $\wedge^k D_{a(\omega)} \varphi^t_\omega$ 
	as a perturbation of $\wedge^k e^{(\alpha + \Delta) t}$ in view of Proposition \ref{h1:attractor}.
	Below, given $\bi = (i_1 , \cdots, i_k)$ we write $\Lambda_\bi := \sum_{j =1 }^k (\alpha - \lambda_{\bi_j})$. 
	\begin{lemma}\label{lem:Qgeneral}
		Assume $\alpha - \lambda_k > 0$. Let $T > 0$ and $\varepsilon > 0, \varepsilon \ll \alpha - \lambda_k$ be fixed, and let $\omega \in \Omega$ be a noise path with the property that the nonlinear term $B_\omega^t := -3 a_\alpha(\theta^t \omega)^2$ satisfies 
		\[
		\| B_\omega^t\|_{\cC(\cO)} \leq \varepsilon
		\]
		for all $t \in [0,T]$. 
		Finally, assume ${\bf v}_0 = v_0^1 \wedge \cdots \wedge v_0^k \in \wedge^k L^2$ satisfies $Q^{(k)}_{\delta}({\bf v}_0) > 0$ where $\delta > 0$ satisfies
		\[
		\ve (1 + \delta) k \leq \Lambda_{{\bf i}_0} - \Lambda_\bi - \frac{\ve (1 + \delta) k}{\delta} \, ,
		\] 
		for all $\bi \neq \bi_0$.
		% What is the difference between $\bi$ and $\bi_0$ here? We should explain meaing of that condition.
		Under these conditions, the $k$-blade ${\bf v}_t := \wedge^k D_{a(\omega)} \varphi^t_\omega ({\bf v}_0)$, corresponding with the time-$t$ solutions $v_t^j = D_{a(\omega)} \varphi^t_\omega(v_0^j)$ to the first variation equation \eqref{eq:var_eq_attract}, satisfies
		\[
		\frac12 \frac{\txtd}{\txtd t} Q^{(k)}_{\delta}({\bf v}_t) \geq \left( \Lambda_{\bi_0} - \frac{ \varepsilon(1 + \delta) k }{\delta}
		\right) Q^{(k)}_{\delta}({\bf v}_t) \, .
		\]
	\end{lemma}
	Lemma \ref{lem:topFTLE} above is a special case of Lemma \ref{lem:Qgeneral} with $k = 1$ and 
	$\delta = \sqrt{\ve}$. Fixing this value of $\delta$, we see that parallel to the argument presented around equation~\eqref{est:Qdeltat}, $Q^{(k)}_\delta ({\bf v}_0) > 0$ implies
	${\bf v}_t = \wedge^k D_{a(\omega)} \varphi^t_\omega ({\bf v}_0)$ satisfies
	\[
	Q^{(k)}_\delta({\bf v}_t) \geq Q^{(k)}_\delta({\bf v}_0) \exp \{2 t (\Lambda_{\bi_0} - 2 k \delta)\} \, .
	\]
	In particular, for $M > 1$ we have that if $Q^{(k)}_{\delta / M} ({\bf v}_0) > 0$, then 
	\[
	\| {\bf v}_t \|^2 \geq \frac{M-1}{M+\delta} \exp \{2 t (\Lambda_{\bi_0} - 2 k \delta)\} \| {\bf v}_0 \|^2\, .
	\] 
	%hence
	%\[
	%\prod_{i =1 }^k \sigma_i(D_{a(\omega)} \varphi^t_\omega) \geq \frac{M-1}{M+\delta} \exp  t (\Lambda_{\bi_0} - 2 k \delta) \, . 
	%\]

 	The proof of Proposition~\ref{prop:FTvol}(b) is now complete on taking $M$ sufficiently large and $\delta$ sufficiently small, and appealing to Proposition \ref{h1:attractor} to ensure $\|B_\omega^t\|_{\cC(\cO)}$ is sufficiently small along the time window $[0,T]$ (c.f. the end of the proof of Proposition \ref{prop:FTLE_1}(b)).  
	
	%\todo[inline]{Is it not more meaningful to directly write the result in terms of sums of FTLEs or products of singular values respectively (as given above), expressing $k$-dimenbsional volume expansion?}

	\subsubsection*{Proof of Lemma \ref{lem:Qgeneral}}
	
	To start, note the unconditional estimate
	\[
	| Q^{(k)}_\delta (v_1 \wedge \cdots \wedge v_k , w_1 \wedge \cdots \wedge w_k)|  \leq (1+\delta) \| v_1 \wedge \cdots \wedge v_k \| \| w_1 \wedge \cdots \wedge w_k\| \, . 
	\]
	Let $v_i = v_i(t) = D_{a(\omega)} \varphi^t_\omega(v_i(0))$.  Assuming $B = B^t_\omega$ is such that $\| B \|_V \leq \ve$, it follows that $\| B v\| \leq \varepsilon \| v \|$ for any $v \in H$. Therefore, 
	\begin{align*}
	\frac12 \frac{\txtd}{\txtd t} Q_\delta^{(k)} (v_1 \wedge \cdots \wedge v_k) & = 
	\sum_{j = 1}^k 
	Q_\delta^{(k)} (v_1 \wedge \cdots \wedge v_k,  v_1 \wedge \cdots \wedge \dot v_j \wedge \cdots 
	\wedge  v_k) \\
	& = 
	\sum_{j = 1}^k 
	Q_\delta^{(k)} (v_1 \wedge \cdots \wedge v_k,  v_1 \wedge \cdots \wedge (\alpha + \Delta + B) v_j \wedge \cdots 
	\wedge  v_k) \\
	& \geq \sum_{j = 1}^k 
	Q_\delta^{(k)} (v_1 \wedge \cdots \wedge v_k,  v_1 \wedge \cdots \wedge (\alpha + \Delta ) v_j \wedge \cdots 
	\wedge  v_k) \\
	& - \ve (1+\delta) k  \| v_1 \wedge \cdots \wedge v_k\|^2\,.
	\end{align*}
	We can write 
	\[
	v_1 \wedge \cdots \wedge v_k = \sum_{\bf i} v_{\bf i} e_{\bf i},
	\]
	such that
	\[
	Q^{(k)}_\delta(v_1 \wedge \cdots \wedge (\alpha + \Delta) v_j \wedge \cdots \wedge v_k) =\sum_{{\bf i}, {\bf i'}} v_{\bf i} v_{\bf i'} (\alpha - \lambda_{i_j'}) Q^{(k)}_\delta (e_{\bf i}, e_{\bf i'})\,.
	\]
	For the summands we have
	\[
	Q^{(k)}_\delta(e_{\bf i}, e_{\bf i'}) = \begin{cases}
	\delta & {\bf i} = {\bf i'} = {\bf i}_0 ,\\
	-1 & {\bf i} = {\bf i'} \neq {\bf i}_0, \\
	0 & \text{else},
	\end{cases}
	\]
	which implies
	\[
	Q^{(k)}_\delta(v_1 \wedge \cdots \wedge (\alpha + \Delta) v_j \wedge \cdots \wedge v_k) = \delta (\alpha - \lambda_j) v_{{\bf i}_0}^2 - \sum_{{\bf i} \neq {\bf i}_0 } v_{\bf i}^2 (\alpha - \lambda_{i_j})\,
	\]
	such that we obtain
	\begin{align*}
	\sum_{j = 1}^k Q^{(k)}_\delta(v_1 \wedge \cdots \wedge (\alpha + \Delta) v_j \wedge \cdots \wedge v_k)
	= \delta \Lambda_{\bi_0} v_{{\bf i}_0}^2 
	- \sum_{{\bf i} \neq {\bf i}_0} \Lambda_\bi v_{\bf i}^2 \,.
	\end{align*}
	%Recall that $\Lambda_{\bi_0} > 0$ and that if ${\bf i} \neq {\bf i}_0$, then
	%\[
	%\Lambda_{{\bf i}} < \Lambda_{{\bf i}_0} \,. 
	%\]
	%Therefore, 
	%\[
	%\sum_{j = 1}^k Q^{(k)}_\delta(v_1 \wedge \cdots \wedge (\alpha + \Delta) v_j \wedge \cdots \wedge v_k)
	%\geq  \Lambda_{{\bf i}_0} \left(\delta v_{{\bf i}_0}^2 - \sum_{{\bf i} \neq {\bf i}_0} v_{\bf i}^2 \right) 
	%\]
	Altogether, we have
	\begin{align*}
	\frac12 \frac{\txtd}{\txtd t} Q^{(k)}_\delta(v_1 \wedge \cdots \wedge v_k ) 
	&  \geq \delta \Lambda_{\bi_0}  v_{\bi_0}^2 - \sum_{\bi \neq \bi_0} \Lambda_{\bi} v_\bi^2 - \ve (1 + \delta) k \sum_{\bi} v_\bi^2 \\
	& = \delta \left( \Lambda_{\bi_0} - \frac{\ve (1 + \delta) k}{\delta} \right)  v_{\bi_0}^2 
	- \sum_{\bi \neq \bi_0} (\Lambda_{\bi} + \ve(1 + \delta) k) v_\bi^2 \, .
	\end{align*}
	Therefore, we may conclude
	\[
	\frac12 \frac{\txtd}{\txtd t} Q^{(k)}_\delta(v_1 \wedge \cdots \wedge v_k ) \geq \left( \Lambda_{{\bf i}_0}  - \frac{\ve (1 + \delta) k}{\delta} \right)  v_{\bi_0}^2  Q_\delta^{(k)}(v_1 \wedge \cdots \wedge v_k),
	\]
	as long as 
	\[
	\ve (1 + \delta) k \leq \Lambda_{\bi_0} - \Lambda_\bi - \frac{\ve (1 + \delta) k}{\delta} 
	\]
	for all $\bi \neq \bi_0$. This finishes the proof of Lemma~\ref{lem:Qgeneral}.
	% It is interesting that everything comes together well here at the end. Maybe we can give some intuition?

	\section{Outlook}\label{sec:outlook}

%	\todo[inline]{If we add this here we need to refer to this section when we discuss the structure of the paper.}

In this paper we considered a scenario where stochastic driving destroyed 
a bifurcation and showed that some ``signature'' of the bifurcation persisted
in the form of a positive FTLE on finite timescales and with positive probability. 
The results in this paper suggest several possible areas of future work, some of which
we list below: 

\medskip

\noindent {\bf Broader class of models. } 
It should be possible to extend the techniques of this paper to 
a broader class of SPDEs exhibiting synchronization phenomena ``destroying''
bifurcations, e.g., the class of higher-order models exhibited in \cite{Bloemker}, 
or systems undergoing Hopf bifurcations. 

\medskip

\noindent {\bf Quantitative estimates on FTLE. }
Unaddressed by our work is the concrete value of the 
probability of `seeing' a positive FTLE on a given timescale
for a statistically stationary initial condition. In view of the convergence of the FTLE
to the asymptotic Lyapunov exponent with probability 1, this kind of quantitative information
amounts to a large deviations estimate. 
%Typically, large deviations principles for 
%SODE and some SPDE can be approached
%using the Feynman-Kac semigroup formalism. 
This is naturally tied to the ergodic properties of the Markov process $(u_t, \hat v_t)$, where
$u_t$ is a solution to the SPDE, $v_t$ is a solution to the first variation equation, and 
$\hat v_t = v_t / \| v_t \|$, see e.g. \cite{arnold1984formula} in the context of SODE. However, 
the ergodic properties of $(u_t, \hat v_t)$ are difficult to study due to the normalization by 
$\| v_t\|$ and the difficult-to-rule-out possibility that the asymptotic Lyapunov exponent is $-\infty$. A rigorous proof of the fact that the asymptotic Lyapunov exponent is negative was recently obtained in~\cite{GT22}.
However, a better understanding of the ergodic properties of this process remains an interesting open problem for a large class
of systems.

\medskip

 \noindent {\bf Arbitrary initial data. } 
It is an interesting question, outside the scope of this work, to provide lower bounds on the FTLE $\Lambda_1$ for a trajectory initiated away from $0$ (large initial data). Suppose, for instance, that one could arrange so that the stationary measure $\rho$ is fully supported in $V_\gamma$, and that the corresponding Markov semigroup is strong / ultra Feller \cite{seidler2001note}.  According to~\cite[\S 8.3.1]{cerrai:ergodicity} we know that these properties hold in $H$ under our assumptions.
It would then hold that \emph{any} initial data $u_0 \in V_\gamma$ enters a $V_\gamma$-small neighborhood of $0$ given enough time (with probability 1). This, in conjunction with the arguments in Section \ref{sec:proofs}, would imply 
the following: for all $u_0 \in V_\gamma$ and $T > 0$, there exists $t > 0$ and a set $\mathcal A \in \mathcal F_0^{t + T}$ such that for $\omega \in \mathcal A$, 
\[
\| D_{u_t} \varphi^s_{\theta^t \omega} \| \approx e^{s (\alpha - \lambda_1)} \quad \text{ for all } \quad s \in [0,T] \, , 
\]
along the lines of Theorem \ref{thm:topLEEst}. A significantly harder question, however, is to provide an estimate of $\| D_{u_0} \varphi^s_{\omega}\|$, initiated at time zero, for potentially large initial data $u_0$ and for times $s > T(u_0)$. While it should be possible to steer the trajectory of $(u_t)$ close to $0$ and apply the arguments of Section \ref{sec:proofs}, what is missing is an argument to bound $D \varphi^t_\omega$ from below during the `transient' time period before $u_t$ has been `steered' toward $0$.

%	\begin{itemize}
%		\item Stochastic driving destroys bifurcations; here we have recovered some
%		signature of them as transient features in the finite-time dynamics of the system
%		\item Several possible directions for future work along these lines: 
%			\begin{itemize}
%				\item Extend our results to broader classes of parametrized SPDEs exhibiting synchronization, e.g., , or to systems undergoing a Hopf bifurcation
%				\item  Additional interesting question not addressed in this work: estimates of asymptotic Lyapunov exponents (Remark 1.2). E.g., interesting what happens as bifurcation parameter $\alpha \to \infty$, since number of ``activated'' modes near zero grows to infinity
%				\item quantify probability of seeing positive FTLE along longer and longer timescales: difficult large deviations problem related to moment Lyapunov exponents. Likely to be a very challenging problem. One issue to overcome: 
%				ergodic theory of ``projectivized'' process $(u_t, \hat v_t)$ tracking
%				a solution $u_t$ to the SPDE and $\hat v_t = v_t / \| v_t\|$, a unit vector in  $H$. 
%			\end{itemize}
%		\item
%	\end{itemize}
	
%	\begin{itemize}
%			\item Asymptotic Lyapunov-exponents, recall Remark 1.2. 
%		\item 
%	\end{itemize}

	\section*{Acknowledgment}

\noindent A. Blumenthal was supported by National Science Foundation grant DMS-2009431.
	
\noindent M. Engel has been supported by Germany's Excellence Strategy -- The Berlin Mathematics Research Center MATH+ (EXC-2046/1)-- project ID: 390685689 (subprojects AA1-8 and AA1-18). He further thanks the German Research Foundation (DFG) for support via the SPP2298 and the CRC1114.
 
\noindent 
The authors thank Christian Kuehn for pointing to references \cite{barret2015sharp, berglund2013sharp} and to the referees for the numerous valuable suggestions. The first author thanks Sam Punshon-Smith for valuable discussions leading to the approach taken in Section \ref{sec:proofs}.

	%%%%%%%%%%%%%
	
	\appendix
	
	\section{Appendix}\label{appendix}
	
%	\todo[inline]{Change ``v'' to some other letter, e.g., $\tilde u$ }
	
	\subsection{Proof of Proposition~\ref{h1:attractor}}
	
 Proposition \ref{h1:attractor} will be deduced from the following. 
	\begin{lemma}\label{lem:propagationReg}
	For any $T > 0$ and $\varepsilon > 0$ there is an event $\mathcal A_T \in \mathcal F_{0}^T$ of positive probability and a real number $\eta = \eta(\varepsilon, T) > 0$ such that for any $u_0 \in V_\gamma$ with $\| u_0\|_{V_\gamma} < \eta$, we have that $\| u_t\|_{V_\gamma} < \varepsilon$ for all $t \in [0,T]$ and all $\omega \in \mathcal A_T$. 
	\end{lemma}
	
	\begin{proof}[Proof of Proposition \ref{h1:attractor} assuming Lemma \ref{lem:propagationReg}]
		By Proposition \ref{prop:statmeas}, the process $(u_t)$ admits a unique stationary measure $\rho$ on $H$ with $\rho(V_\gamma) = 1$, hence $a_\alpha(\omega) \in V_\gamma$ with probability 1. Let $\varepsilon, T > 0$ be fixed let and $\eta = \eta(\varepsilon, T) > 0, \mathcal A_T \in \mathcal F_0^T$ be as in Lemma \ref{lem:propagationReg}. Local positivity of $\rho$ ensures $\bP(\mathcal A_-) > 0$, where 
		\[ \mathcal A_- = \{ \| a_\alpha \|_{V_\gamma} < \eta \} \, .\] 
		Let now $\mathcal A = \mathcal A_- \cap \mathcal A_T$, and note that since $\mathcal A_- \in \mathcal F_{-\infty}^0$, we have that $\mathcal A_-, \mathcal A_T$ are independent, hence
		\[
		\bP(\mathcal A) = \bP(\mathcal A_- ) \cdot \bP(\mathcal A_T) > 0 \,,
		\]
		which completes the proof. 
	\end{proof}
	
	\begin{proof}[Proof of Lemma \ref{lem:propagationReg}]
	
%{\color{red} stopped editing here; i don't follow the cutoff argument - AB }

The main idea is to exploit regularizing properties of analytic semigroups. % together with a suitable version of the Gronwall inequality~\cite{henry}.
	Recalling the SPDE~\eqref{spde_intro}
	\begin{align}\label{spde}
	\txtd u &= [\Delta u + \alpha u \underbrace{- u^3}_{:=f(u)}]~\txtd t + \txtd W_t\nonumber\\
	%& = [\Delta u - \lambda_1 u +\alpha u ]~\txtd t + \underbrace{ (\lambda_1 u - u^3)}_{:=f(u)}~\txtd t + \txtd W(t)\nonumber\\
	& =  [\Delta u +\alpha u ]~\txtd t + f(u)~\txtd t + \txtd W_t,
	\end{align}
	we subtract the Ornstein-Uhlenbeck process, i.e. the solution of the linear SPDE
	\begin{align*}
	\txtd z = [\Delta z  +\alpha z]~\txtd t + \txtd W_t.
	\end{align*}
	This solution is given by the convolution
	\begin{align*}
	z(t) =\int_0^t T(t-r)~\txtd W_r,
	\end{align*}
	where $(T(t))_{t\geq 0}$ is the shifted heat semigroup with $\alpha$, i.e.~$T(t):=e^{(\Delta +\alpha)t}$. 
	According to Lemma~\ref{lem:stochConvolution} this belongs to $C([0,T];V_\gamma$ $\mathbb{P})$-a.s. Therefore the set
	$$ \cA_T:=\Big\{  \omega\in\Omega : \sup\limits_{t\in[0,T]} \|z(t)\|_{V_\gamma} \leq \eta\Big\}\in \cF^T_0$$
	has positive probability. This is sufficient for our aims since we are only interested in a finite-time statement. We further fix $\omega\in \cA= \cA_{-}\cap\cA_T$ as in Proposition~\ref{h1:attractor} and study on $V_\gamma$ the PDE with random non-autonomous coefficients
	\begin{align}\label{rpde}
	\txtd  {\tilde u} = [\Delta \tilde{u}   + \alpha \tilde{u} ]~\txtd t + f(\tilde{u}+z)~\txtd t.
	\end{align}
(Since $(T(t))_{t\geq 0}$ is an analytic semigroup, $\tilde{u}$ is differentiable on $V_\gamma$ for $t>0$). 
We assume that the initial data	 $\tilde{u}_0:=\tilde{u}(0)=a_\alpha(\omega)$.  
We now prove that for any given $\varepsilon > 0$ and $\omega\in\cA$ we have
		\begin{align}\label{aim}
		\|a_{\alpha}(\theta^s\omega)\|_{V_\gamma}\in(0,\varepsilon), \text{ for} ~s\in[0,T].%\text{ and } \omega\in \cA
		\end{align}
	%	for a $\cF^T_{-\infty}$-measurable set $\cA$ with $\bP(\cA)>0$.	
	To this aim we use the fact the semigroup $(T(t))_{t\geq 0}$ acts on all function spaces $V_\gamma$ together with the classical estimate%~\cite{Amann_1995}
		\begin{align}
		\|T(t)\|_{\mathcal B (V_\gamma)} &\leq e^{(-\lambda_1+\alpha)t},~~t>0 \label{T_V_est}.%\\
	%	\|T(t)\|_{\mathcal B(H,V)} &\leq t^{-\gamma} e^{(-\lambda_1+\alpha)t},~~t>0 \label{T_H_V_est}
		\end{align}
The cubic nonlinearity $f:V_\gamma\to V_\gamma$ is locally Lipschitz for $\gamma\in(\frac{1}{4},\frac{1}{2})$, i.e. there exists a constant $\widetilde{l}:=\widetilde{l}(\|\tilde{u}_1\|_{V_\gamma}, \|\tilde{u}_2\|_{V_\gamma})$ such that
		\begin{align*}
		\|f(\tilde{u}_1) -f(\tilde{u}_2)\|_{V_\gamma} \leq \tilde{l} \|\tilde{u}_1-\tilde{u}_2\|_{V_\gamma}, ~~\text{ for }~ \tilde{u}_1,\tilde{u}_2\in B\subset V_\gamma,
		\end{align*}
		where $B$ is a bounded subset of $V_\gamma$. Moreover there exists an increasing function $a:\mathbb{R}^+\to\mathbb{R}^+$ such that
		\begin{align}\label{en:f}
			\langle F(y+w), y \rangle_{V_\gamma} \leq a(\|w\|)(1+\|y\|),~\text{ for all } y,w\in V_\gamma,
 		\end{align}
 	see~\cite[(H3),~p.~130]{stannat} and~\cite[Section 7.2.1]{DaPratoZabczyk}. Regarding this together with $\|\cdot\|_H\leq \|\cdot\|_{V_\gamma}$ further results in
 	\begin{align*}
 		\frac{1}{2}\frac{\txtd}{\txtd t} \|\tilde{u}(t)\|^2_{V_\gamma} &\leq \langle \Delta \tilde{u}(t) +\alpha \tilde u(t), \tilde{u}(t) \rangle_{V_\gamma} +\langle F(\tilde{u}(t) +z(t) ), \tilde{u}(t) \rangle_{V_\gamma}\\
 		& \leq  (-\lambda_1+\alpha) \|\tilde{u}(t)\|^2_{V_\gamma} +a(\|z(t)\|_{V_\gamma}) (1+\|\tilde{u}(t)\|_{V_\gamma})\\
 		& \leq (-\lambda_1+\alpha +a(\|z(t)\|_{V_\gamma}) )\|\tilde{u}(t)\|^2_{V_\gamma} +a(\|z(t)\|_{V_\gamma}).
 	\end{align*}
 Here we estimated $\|\tilde{u}(t)\|_{V_\gamma} \leq \|\tilde{u}(t)\|^2_{V_\gamma}$, since if $\|\tilde{u}(t)\|_{V_\gamma}\leq 1$, then the statement automatically holds.
 Gronwall's inequality now entails
 \begin{align*}
 	\|\tilde{u}(t)\|^2_{V_\gamma} & \leq e^{(-\lambda_1+\alpha) t +\int_0^t a(\|z(s)\|_{V_\gamma})~\txtd s } \|\tilde{u}_0\|_{V_\gamma} +\int_0^t e^{(-\lambda_1 +\alpha)(t-s) +\int_s^t a(\|z(r)\|_{V_\gamma})~\txtd r } a(\|z(s)\|_{V_\gamma})~\txtd s\\
 	& \leq e^{(-\lambda_1 +\alpha + a(\eta))t } \|\tilde{u}_0\|_{V_\gamma} +a(\eta) \int_0^t e^{(-\lambda_1+\alpha+a(\eta))(t-s)}~\txtd s,
 \end{align*}
 	where we used the fact that $a$ is increasing and that $\|z(s)\|_{V_\gamma} \leq \sup\limits_{s\in[0,T]} \|z(s)\|_{V_\gamma}\leq \eta $ for $\omega\in\cA$. 
 		Regarding that $\|u(t)\|_{V_{\gamma}} \leq \|\tilde{u}(t)\|_{V_\gamma} +\|z(t)\|_{V_\gamma}$, $u_0=a_\alpha(\omega)$ and consequently $u(t)=a_\alpha(\theta^t\omega)$, the statement follows.
	\end{proof}

	\newpage
	\bibliographystyle{abbrv}
	
	\bibliography{Bibliography2.bib}

\end{document}